%% file: main.tex
\documentclass{article}
\usepackage{cmd_pkg}

\title{The existence of minimizers for an isoperimetric problem with Wasserstein penalty term in unbounded domains}
\author{Qinglan Xia and Bohan Zhou}
\newcommand{\Addresses}{{% additional braces for segregating \footnotesize
  \bigskip
  \footnotesize

  QINGLAN XIA, \textsc{Department of Mathematics, University of California at Davis}\par\nopagebreak
  \textit{E-mail address}: \texttt{qlxia@math.ucdavis.edu}

  \medskip

  BOHAN ZHOU, \textsc{Department of Mathematics, University of California at Davis}\par\nopagebreak
  \textit{E-mail address}: \texttt{bhzhouzhou@math.ucdavis.edu}

}}
\date{\today}

\begin{document}

\maketitle
\begin{abstract}
In this article, we consider the (double) minimization problem
\[\min\left\{P(E;\Omega)+\lambda W_p(\Leb^d\measurerestr E,\Leb^d\measurerestr F):~E\subseteq\Omega,~F\subseteq \mathbb{R}^d,~\abs{E\cap F}=0,~ \abs{E}=\abs{F}=1\right\},\]
where $p\geqslant 1$, $\Omega$ is a (possibly unbounded) domain in $\R^d$, $P(E;\Omega)$ denotes the relative perimeter of $E$ in $\Omega$ and $W_p$ denotes the $p$-Wasserstein distance. When $\Omega$ is unbounded and $d\geqslant 3$, it is an open problem proposed by Buttazzo, Carlier and Laborde in the paper \textit{On the Wasserstein distance between mutually singular measures}. We prove the existence of minimizers to this problem when $\frac{1}{p}+\frac{2}{d}>1$, $\Omega=\R^d$ and $\lambda$ is sufficiently small. 
\end{abstract}


\section{Introduction}

In this paper, we consider an open question left by Buttazzo, Carlier and Laborde in \cite{buttazzo2017wasserstein}. Let $\Omega$ denote a (possibly unbounded) open domain in $\R^d$ with volume $\abs{\Omega}>1$. For $\lambda\geqslant 0$ and $p\geqslant 1$, authors of \cite{buttazzo2017wasserstein} consider the following (double) minimization problem:
\begin{align}
\label{eq: Buttazzo}
    \min\left\{P(E; \Omega)+\lambda W_p(\Leb^d\measurerestr E,\Leb^d\measurerestr F):~E\subseteq\Omega,~F\subseteq \mathbb{R}^d,~\abs{E\cap F}=0,~ \abs{E}=\abs{F}=1\right\}
\end{align}
where $P(E;\Omega)$ denotes the relative perimeter of $E$ in $\Omega$ (\cite{maggi2012sets}) and $W_p$ denotes the $p$-Wasserstein distance (\cite{Villani2003Topics}) between probability measures.

As studied in \cite{Peletier2009Partial,Lussardi2014Variational}, this type of problem arises from some biological models of bi-layer membranes. To study such an isoperimetric problem with Wasserstein penalty term, to our best knowledge, most literature assume that $\Omega$ is bounded and $F$ is given. For instance, to model materials cracking problem, the first author in \cite{xia2005regularity} studies the existence and regularity when the second term is replaced by $\lambda W_p^p(\Leb^d\measurerestr E,\sigma \Leb^d\measurerestr \Omega)$, where $\Omega$ is bounded and $\abs{E}=\sigma\abs{\Omega}$. Milakis in \cite{Milakis2006Regularity} studies an analogous problem for $W_2^2(\Leb^d\measurerestr E, \Leb^d\measurerestr F)$ as the second term when $\Omega$ is a bounded smooth domain and $F$ is given. In other scenarios, for fixed $F$, if one replaces the perimeter term by some functional on $E$ and adopts $W_2^2(\Leb^d\measurerestr E, \Leb^d\measurerestr F)$, such a variational problem corresponds to the \textit{Jordan-Kinderlehrer-Otto} (JKO) scheme (\cite{JKO1998Fokker}), which can be regarded as a gradient flow under Wasserstein metric (see the review paper \cite{Santambrogio2017GF}). This leads to many interesting problems and applications (see \cite{DePhilippis2016BV,Santambrogio2018crowd,di2019jko}). When $\Omega$ is unbounded, besides the classical Euclidean isoperimetric problem (see \cite{morgan1996surface}) and the founding work by Almgren in \cite{Almgren1976Existence} on minimizing clusters problem, Kn\"{u}pfer and Muratov in \cite{Knupfer2013isoperimetric,Knupfer2014Isoperimetric} study an isopermetric problem with a non-Wasserstein term. The penalty term there are generated by a kernel given by an inverse power of the distance. Other related work might be found in \cite{Figalli2015Isoperimetry}.

In \cite{buttazzo2017wasserstein} Buttazzo et al. prove the existence of minimizers to \eqref{eq: Buttazzo} for the following cases when $\lambda>0$:
\begin{itemize}
    \item For any $d$, when $\Omega$ is bounded, the minimization problem \eqref{eq: Buttazzo} admits a solution.
    \item For $d=2$ and $\Omega=\R^2$, the minimization problem \eqref{eq: Buttazzo} admits a solution.
    \item For $d=1$, a solution can be constructed by disjoint equal sub-intervals, whose number depends on $\lambda$.
\end{itemize}

Their proof for the case $d=2$ and $\Omega=\R^2$ relies on the fact that for a connected set, its diameter is bounded by its perimeter, which only holds for dimension two. Therefore the existence to such a minimization problem is still open for a unbounded domain $\Omega$ of dimension more than two.

In this article, we adopt a new approach that is valid for every dimension $d$. It provides the existence result in every dimension for small $\lambda$.

\begin{theorem}\label{thm: minEset}
Suppose $p\geqslant 1$, $d\geqslant 1$ with $\frac{1}{p}+\frac{2}{d}>1$ and $\Omega=\R^d$, there exists $\lambda_0=\lambda_0(d,p)>0$, such that for any $0<\lambda\leqslant \lambda_0$, the minimization problem
\begin{equation}
    \label{eq: minEset}
    \min\left\{P(E)+\lambda W_p(\Leb^d\measurerestr E, \Leb^d\measurerestr F):~E,F\subseteq \mathbb{R}^d,~\abs{E\cap F}=0,~ \abs{E}=\abs{F}=1\right\}
\end{equation}
admits a solution.
\end{theorem}

For $d\geqslant 3$ and $\Omega=\R^d$, the main difficulty is that we only have compactness of sets of locally finite perimeter. As a consequence, the limit set of any minimizing sequence with respect to convergence in measure may not satisfy the volume constraint. To overcome this obstacle, we adopt the following strategy:

\begin{itemize}
    \item \textbf{Equivalent formulation in a volume parameter $m$.} To normalize the parameter $\lambda$ in the problem \eqref{eq: minEset}, we apply scaling arguments and obtain an equivalent formulation in the problem \eqref{problem: m} with a volume parameter $\abs{E}=m$.
    
    \item \textbf{Existence of a minimizing sequence of bounded sets.} We prove in \prettyref{thm: alter_bound_seq} that there exists a minimizing sequence of \textit{bounded} sets to the problem \eqref{problem: m}. In our proof, we use a ``covering-packing'' technique: We first cover the majority of the set $E$ by a prescribed number of balls with same radius in Proposition \prettyref{prop: perimeter_inc}. Here we use the so-called \textit{Nucleation Lemma} in \cite{maggi2012sets}, which is a tool from Almgren's seminal paper \cite{Almgren1976Existence} for minimizing clusters problem. Then we pack all balls into a ball of prescribed radius in \prettyref{thm: rearrange}. Applying this ``covering-packing" technique to any given minimizing sequence, we obtain an alternative minimizing sequence of bounded sets as desired. Now, by using the known result for $\W_p(E)\defeq\min_F W_p(E,F)$ on any bounded set $E$, we express the double minimizing problem \eqref{problem: m} into an equivalent single minimizing problem \eqref{problem: T}: $\textrm{Minimize~}P(E)+\W_p(E)$ among all bounded sets $E$ with $\abs{E}=m$.
    
    \item \textbf{Existence of a minimizing sequence of uniformly bounded sets for small volume.} To apply the direct method of calculus of variations, we further require uniform boundedness. When the volume is small, in \prettyref{thm: comp_bdset} we are able to find a minimizing sequence of uniformly bounded sets to the problem \eqref{problem: T}, through a non-optimality criterion in Proposition \prettyref{prop: nonoptimality1}. Our work is inspired by the seminal work of Kn\"{u}pfer and Muratov in \cite{Knupfer2014Isoperimetric} for an isoperimetric problem with a competing non-local term in unbounded domains.
\end{itemize}
\begin{remark}
It is interesting to compare the non-local functional $V(E)=\int_E\int_E \frac{1}{\abs{x-y}^{\alpha}}\diff{x}\diff{y}$ for $\alpha\in (0,d)$ in \cite{Knupfer2014Isoperimetric} with the non-local Wasserstein term $\W_p(E)=\min_{F}W_p(E,F)$. Both non-local terms behave like repulsive effects with respect to the set itself. The non-local term in \cite{Knupfer2014Isoperimetric}, among which the \textit{Coulombic repulsion} is a special case, is in an exact integral form. Thus it has a natural advantage to compare the functional between different sets. In opposite, the Wasserstein term consists of a minimizing process. It requires to minimize among all disjoint sets of equal volume, and to minimize among all admissible transport plans, which bring novel obstacles.
\end{remark}

%And, it might be helpful to remark the readers that the equivalence between \prettyref{thm: minEset}, \prettyref{thm: minGset} and \prettyref{thm: minTset}. In fact, we prove the equivalence and the existence of the minimizers to \prettyref{thm: minTset} for small volume sets, which corresponds to the existence of the minimizers to \prettyref{thm: minEset} for small parameter $\lambda$.

The remaining of the paper is organized as follows: in \prettyref{sec: not} we introduce the notations throughout the paper. In \prettyref{sec: pre} we recall some basic definitions in geometric measure theory, with an emphasis on the theory about sets of finite perimeter and optimal transport theory. In \prettyref{sec: pre_re} we reformulate the problem \eqref{eq: minEset} into the problem \eqref{problem: m}. In \prettyref{sec: Was_functional}, we introduce the Wasserstein functional $\W_p(E)$ on any bounded Lebesgue measurable set $E$ and study its properties. In \prettyref{sec: bnd_seq}, we prove the existence of a minimizing sequence of bounded sets to the problem \eqref{problem: m}, by which we reformulate again the problem \eqref{problem: m} into the problem \eqref{problem: T}. In \prettyref{sec: small_vol}, for small volume sets, we prove the existence of a minimizing sequence of uniformly bounded sets, and use it to prove the existence of minimizers for the problem \eqref{problem: T}.

\section{Notations}\label{sec: not}
We use the following notations below throughout the paper.

\begin{longtable}{l|l} 

$B(x,r)$ or $B_r(x)$ & Open $d$-ball centered at $x$ of radius $r$ in $\R^d$.\\
$\omega_d$ & the volume of unit $d-$ball.\\
$\ell_d=(\omega_d)^{-1/d}$ & The radius of $d-$ball of volume 1.\\
%$\mathfrak{C}^m$ & the collection of Lebesgue measurable sets with volume $m$.\\
$C_1\sqcup C_2$ & Disjoint union of sets $C_1$ and $C_2$.\\
$C_1\Delta C_2=(C_1\setminus C_2)\cup (C_2\setminus C_1)$ & Symmetric difference of sets $C_1$ and $C_2$.\\
$rE=\left\{rx: x\in E\right\}$ & Re-scaling of a set $E$.\\
$E+a=\left\{x+b: x\in E\right\}$ & Translation of a set $E$.\\
$\Id_E$ & The characteristic function of set $E$.\\
$\Leb^d$ or $\diff{x}$ & Lebesgue measure.\\
$\Leb^d\measurerestr E$ & $d-$Lebesgue measure restricted on a set $E$.\\
$\abs{E}=\Leb^d(E)$ & Volume ($d-$Lebesgue measure) of a set $E$.\\
$\Phi_{\#}\mu$ & Push-forward of measure $\mu$ by the mapping $\Phi$. \\
$\displaystyle\textrm{dist}(x,E)=\inf_{y\in E} \textrm{dist}(x,y)$ & Distance between a point $x$ and a set $E$ in $\R^d$.\\ 
$\mathcal{P}_c(\R^d)$ & the class of all probability measures on $\R^d$ with compact support.\\
%$\mathcal{M}_p(X)$ & the class of all Radon measures on $X$ with finite $p-$moments.\\
%$\mu=\mu^a+\mu^c$ & The Lebesgue decomposition of $\mu$, where $\mu^a$ is the absolutely \\
%& continuous part and $\mu^s$ is the singular part.\\
$\mu\perp\nu$ & measures $\mu$ and $\nu$ are mutually singular.\\
%${\mu}^{\perp}$ & The class of measures that are singular to $\mu$.\\
%$\mathcal{P}_c(X)$ & the class of all probability measures on $X$ with compact support.\\
%$\mathcal{P}_p(X)$ & the class of all probability measures on $X$ with finite $p-$moments.\\

\end{longtable}

\section{Preliminaries}\label{sec: pre}
In this section, we first recall related concepts in geometric measure theory with an emphasis on sets of finite perimeter \cite{maggi2012sets} and optimal transport theory \cite{Villani2003Topics, Villani2009Old}. 

\subsection{Sets of finite perimeter}\label{subsec: set_fp}
In this subsection, we closely follow Maggi's book \cite{maggi2012sets}.
\begin{definition}[Set of finite perimeter]
We say that a Lebesgue measurable set $E\subseteq \R^d$ is a set of locally finite perimeter if for every compact set $K\subseteq \R^d$ we have
$$\sup\left\{\int_E \textrm{div~}\phi(x)\diff{x}: \phi\in C_c^1(\R^d;\R^d), \textrm{spt~}\phi \subseteq K, \sup_{\R^d}\abs{\phi}\leqslant 1\right\}<\infty.$$
If the above quantity is bounded independently of $K$, then we say $E$ is a set of finite perimeter.

If $E$ is a set of locally finite perimeter, then there exists a $\R^d-$valued Radon measure $\mu_E$, called the distributional derivative of set $E$, such that
$$\int_{E}\textrm{div~}\phi(x)\diff{x}=\int_{\R^d}\phi\cdot\diff{\mu_E},\qquad\forall \phi\in C_c^1(\R^d;\R^d).$$

The perimeter of $E$ in $\Omega$, denoted by $P(E;\Omega)$, is the variation of $\mu_E$ in $\Omega$, i.e.,
$$P(E;\Omega)\defeq \abs{\mu_E}(\Omega).$$

When $\Omega=\R^d$, we adopt $P(E)$ for simplicity.
\end{definition}

\begin{definition}[Convergence in measure]
Given a sequence $\{E_n\}$ of Lebesgue measurable sets and $E$ in $\R^d$, we say that $E_n$ locally converges to $E$, denoted by $E_n \xrightarrow[]{\textrm{loc}} E$, if 
$$\lim_{n\to\infty} \abs{K\cap (E\Delta E_n)}=0,\qquad\qquad\forall K\subseteq\R^d\textrm{~compact}.$$

We say $E_n$ converges to $E$, denoted by $E_n\to E$, if
$$\lim_{n\to\infty}\abs{E\Delta E_n}=0.$$
\end{definition}

\begin{proposition}[Lower semi-continuity of perimeter]\label{prop: lsc of perimeter}
If $\{E_n\}$ is a sequence of sets of locally finite perimeter in $\R^d$ with
$$E_n \xrightarrow[]{\textrm{loc}} E,\qquad \limsup_{n\to\infty} P(E_n;K)<\infty,$$
for every compact set $K$ in $\R^d$, then $E$ is of locally finite perimeter in $\R^d$, $\mu_{E_n}\xrightharpoonup{*} \mu_E$, and for every open set $\Omega\subseteq \R^d$, we have
$$P(E;\Omega)\leqslant \liminf_{n\to\infty} P(E_n;\Omega).$$
\end{proposition}

\begin{proposition}[Compactness of uniformly bounded sets of finite perimeter]
\label{prop: compact_bdset}
If $r>0$ and $\{E_n\}$ are sets of finite perimeter in $\R^d$, with
\[
\sup_{n} P(E_n)<\infty,\qquad\mathrm{and}\qquad E_n\subseteq B_r,\quad\forall n.
\]
Then there exists a set $E$ of finite perimeter in $\R^d$, such that up to extracting a subsequence (still denoted by $E_n$):
$$E_{n}\to E,\qquad \mu_{E_{n}}\xrightharpoonup{*}\mu_E,\qquad E\subseteq B_r.$$
\end{proposition}

\begin{corollary}[Local compactness of sets of locally finite perimeter]
\label{cor: local_comp}
If $\{E_n\}$ are sets of locally finite perimeter in $\R^d$ with
$$\sup_{h} P(E_h;B_r)<\infty,\qquad \forall r>0.$$
Then there exists a set $E$ of locally finite perimeter, such that up to extracting a subsequence (still denoted by $B_n$):
$$E_{n}\xrightarrow[]{\textrm{loc}}E,\qquad \mu_{E_{n}}\xrightharpoonup{*}\mu_E.$$
\end{corollary}

%\begin{definition}[\red{Set of points of density $t$ of $E$}] 
%Given $E\subseteq \R^d$ and $x\in\R^d$, if the limit
%$$\theta_{d}(E)(x)=\lim_{r\to 0^+}\frac{\abs{E\cap B(x,r)}}{\alpha_d r^d}$$
%exists, it is called the $d$-dimensional density of $E$ at $x$.

%Given $t\in [0,1]$, the set of points of density $t$ of $E$ is defined as
%$$E^{(t)}=\left\{x\in\R^d:\theta_d(E)(x)=t\right\}.$$
%\end{definition}

As in \cite{Figalli2010Mass}, the \textit{isoperimetric deficit} of a set of finite perimeter $E\subseteq\R^d$ is defined by
\begin{align}
    D(E)\defeq\frac{P(E)-P(B_r)}{P(B_r)},
\end{align}
where $B_r$ is a $d-$ball with $\abs{B_r}=\abs{E}$.

The Fraenkel asymmetry of two measurable sets $E_1$ and $E_2$ with $\abs{E_1}=\abs{E_2}$ is defined by
\begin{align}
    \label{eq: Fraenkel}
    \Delta(E_1,E_2)\defeq\min_{x\in\R^d}\frac{\abs{E_1\Delta (E_2+x)}}{\abs{E_1}},
\end{align}
where $E_2+x=\left\{y+x:y\in E_2\right\}$.
\begin{theorem}[\cite{Figalli2010Mass}, Quantitative
isoperimetric inequality]\label{thm: figalli}
There exists a constant $C(d)$ such that for any set $F\subseteq \R^d$ of finite perimeter, we have
\begin{align}
    \label{eq: figalli_isop}
    \Delta(E,B_r)\leqslant C(d)\sqrt{D(E)},
\end{align}
where $B_r$ is a $d-$ball with $\abs{B_r}=\abs{E}$.
\end{theorem}

%\begin{theorem}[\cite{maggi2012sets} Proposition 12.37 and Remark 12.38, Relative isoperimetric inequality]
%For $d\geqslant 2$, $x\in\R^d$ and $r>0$, for any set of locally finite perimeter $E$ such that $\abs{E\cap B(x,r)}\leqslant \frac{1}{2}\abs{B(x,r)}$, if $E\subset B(x,r)$, then there exists a positive constant $C(d)$ such that
%\begin{align}
%    \label{eq: rel_isop}
%    P(E;B(x,r))\geqslant c(d)\min\left\{\abs{E\cap B(x,r)},\abs{B(x,r)\setminus E}\right\}^{1-1/d}.
%\end{align}
%\end{theorem}
%\begin{remark}
%This version does not make sense. Because if $\abs{E\cap B(x,r)}\leqslant \frac{1}{2}\abs{B(x,r)}$, then $\abs{E\cap B(x,r)}\leqslant \frac{1}{2}\abs{B(x,r)}$ and $\abs{B(x,r)\setminus E}\geqslant \frac{1}{2}\abs{B(x,r)}$, leading the min no sense.
%\end{remark}

\subsection{Optimal transport theory}

\begin{definition}[Wasserstein distance]
Let $\mathcal{P}_p(\R^d)\defeq\left\{\mu\in\mathcal{P}(\R^d): \int_{\R^d}\abs{x-x_0}^p\diff{\mu}(x)<+\infty\right\}$ for some point $x_0\in\R^d$. For $\mu,\nu\in \mathcal{P}_p(\R^d)$, the $p$-Wasserstein distance between $\mu$ and $\nu$ is given by
\begin{align}
    \label{eq: Wasserstein}
    W_p(\mu,\nu)\defeq \inf_{\gamma\in \Gamma(\mu,\nu)}\left(\int_{\R^d\times\R^d}\abs{x-y}^p\diff{\gamma}(x,y)\right)^{1/p},
\end{align}
where $\Gamma(\mu,\nu)$ is the collection of the so-called transport plans from $\mu$ to $\nu$, defined by
\begin{align}
    \label{eq: transport_plan}
    \Gamma(\mu,\nu)\defeq\left\{\gamma\in P(\R^d\times\R^d): (\pi_x)_{\#}\gamma=\mu, (\pi_y)_{\#}\gamma=\nu\right\},
\end{align}
where $\pi_x,\pi_y$ denote the projection from $\R^d\times\R^d$ onto each marginal space.

With a slight abuse of notation, given two Lebesgue measurable sets $E,F$ with $\abs{E}=\abs{F}=m$, $W_p(E,F)$ is given by
\[W_p(E,F)\defeq W_p(\Leb^d\measurerestr E,\Leb^d\measurerestr  F)=m^{\frac{1}{p}+\frac{1}{d}}W_p(\Leb^d\measurerestr (m^{-1/d}E),\Leb^d\measurerestr (m^{-1/d}F)).\]
\end{definition}

When $E,F$ are bounded sets of equal volume in $\R^d$ and $p\in [1,\infty)$, it is well-known (\cite{Brenier1991polar}, \cite{Gangbo1996Geometry}, \cite{Ambrosio2003Lecture},  and see also Chapter 1 and Chapter 3 in \cite{Santambrogio2015Optimal}) that there exists an \textit{optimal transport map} $\Phi$ that transports $E$ to $F$, in the sense that $\Phi_{\#}(\Leb^d\measurerestr E)=\Leb^d\measurerestr F$ and
\[W_p(E,F)=\left(\int_{E}\abs{x-\Phi(x)}^p\diff{x}\right)^{1/p}.\]

\subsection{Equivalent formulation of problem \texorpdfstring{$\eqref{eq: minEset}$}{}}\label{sec: pre_re}
For convenience sake, we consider an equivalent formulation of problem \eqref{eq: minEset} by using scaling arguments.

For any $m>0$, denote
\[\mathcal{F}_m:=\left\{(E,F): ~ E, F \subseteq \R^d,~\abs{E\cap F}=0,~ \abs{E}=\abs{F}=m\right\}.\]
Then the minimization problem (\ref{eq: minEset}) becomes 
\begin{equation}\label{problem: lambda}
 \text{ Minimize }\qquad   P(E)+\lambda W_p(E,F)
\qquad\text{ among all } (E,F)\in \mathcal{F}_1.
\end{equation}

Note that for any $(E, F)\in \mathcal{F}_1$ and $r>0$, by the scaling argument, it follows that
$(rE, rF)\in \mathcal{F}_{m}$ for $m=r^d$ with
\begin{align}\label{eq: scaling}
    P(rE)=r^{d-1}P(E)
\quad\textrm{and}\quad W_p(rE, rF)=r^{1+\frac{d}{p}}W_p(E,F).
\end{align}

Now, by setting $m$ to be the number such that
\[\lambda=m^{\frac{1}{p}+\frac{2}{d}-1}\quad\textrm{and~let}\quad r=m^{\frac{1}{d}},\]
we have
\begin{align*}
P(rE)+ W_p(rE,rF)=&r^{d-1}P(E)+r^{1+\frac{d}{p}}W_p(E,F)\\
=&r^{d-1}\left(P(E)+r^{d(\frac{2}{d}+\frac{1}{p}-1)}W_p(E,F)\right)\\
=&r^{d-1}\left(P(E)+\lambda W_p(E,F)\right).
\end{align*}

This gives an equivalent formulation of problem (\ref{problem: lambda}): For any $m\geqslant 0$,
\begin{equation}\label{problem: m} 
 \text{ Minimize }\qquad   P(E)+ W_p(E,F)
\qquad\textrm{among~all~} (E,F)\in \mathcal{F}_m.
\end{equation}
Any solution $(E, F) \in \mathcal{F}_m$ to problem (\ref{problem: m}) corresponds to a solution $(m^{-\frac{1}{d}}E, m^{-\frac{1}{d}}F)\in \mathcal{F}_1$ to problem (\ref{problem: lambda}) (i.e. problem (\ref{eq: minEset})) for
\begin{equation}\label{eq: lambda_to_m}
    \lambda=m^{^{\frac{1}{p}+\frac{2}{d}-1}}
\end{equation}
and
\[P(m^{-\frac{1}{d}}E)+\lambda W_p(m^{-\frac{1}{d}}E,m^{-\frac{1}{d}}F)=m^{\frac{1}{d}-1}\left(P(E)+ W_p(E,F)\right).\]

As a result, to prove \prettyref{thm: minEset}, it is equivalent to prove the following theorem:
\begin{theorem}\label{thm: minGset}
For $\frac{1}{p}+\frac{2}{d}>1$ and $\Omega=\R^d$, there exists $m_0=m_0(d,p)>0$, such that for any $0<m\leqslant m_0$, the minimization problem
\begin{equation}
    \label{eq: minGset}
    \min\left\{P(E)+ W_p(E,F): (E,F)\in \mathcal{F}_m\right\}
\end{equation}
admits a solution.
\end{theorem}

\section{The Wasserstein functional on bounded sets}\label{sec: Was_functional}
\begin{lemma}\label{lem: exist_F}
For any bounded Lebesgue measurable set $E\subseteq \R^d$, there exists a set $F$ such that $(E,F) \in \mathcal{F}_m$ with $m\defeq\abs{E}<\infty$, and 
\[W_p(E,F)=\min\left\{W_p(E,\widetilde{F}): (E,\widetilde{F}) \in \mathcal{F}_m\right\}.\]
\end{lemma}
\begin{proof}
 Without loss of generality, we may assume that $m=1$ by applying the scaling argument \eqref{eq: scaling}. Then the existence of a minimizer $F$ follows from Theorem 3.10 in \cite{buttazzo2017wasserstein} (as re-stated below ) with $\mu=\Leb^d\measurerestr E$, $\phi\equiv 1$ and the observation that $\Leb^d\measurerestr \widetilde{F} \in \mathcal{A}_\phi$ whenever $(E,\widetilde{F})\in \mathcal{F}_m$.
\end{proof}

Here, we restate Theorem 3.10 in \cite{buttazzo2017wasserstein}
with minor modifications in notations:

\begin{theorem}[Theorem 3.10 in \cite{buttazzo2017wasserstein}]\label{thm: buttazzo2017thm310}
For any $\mu\in \mathcal{P}_c(\R^d)$, given a nonnegative integrable function $\phi(x)$ on $\R^d$ with $\int_{\R^d\setminus \textrm{supp}(\mu)}\phi(x)\diff{x}>1$, let $\mathcal{A_{\phi}}$ denote a collection of measures, defined by
\[\mathcal{A}_{\phi}=\left\{\nu\in \mathcal{P}_c: \nu\perp\mu,\frac{\diff{\nu}}{\diff{x}}\leqslant \phi\right\}.\]
Then there exists a set $A$ such that the measure $\nu_0:=\phi \Leb^d\measurerestr {A}$ is in $\mathcal{A}_\phi$ and satisfies
\[W_p(\mu,\nu_0)=\min_{\nu\in A_{\phi}}W_p(\mu,\nu).\]
\end{theorem}

\begin{definition}
For any bounded Lebesgue measurable set $E\subseteq \R^d$ and $p\geqslant 1$, let $m\defeq \abs{E}$ and define the Wasserstein functional on $E$ by
\begin{equation}
  \label{eq: W_fun}
\W_p(E):=\min \{W_p(E,\widetilde{F}):~ (E,\widetilde{F})\in \mathcal{F}_m\}.
\end{equation}
\end{definition}

By the scaling argument \eqref{eq: scaling}, it follows that
\begin{align}\label{eq: W_rescale}
  \W_p(rE)=r^{1+\frac{d}{p}}\W_p(E).
\end{align}

%For any $\W_p$-minimizer $F$ of $E$, by Brenier's theorem, there exists an optimal transport map $\Phi$ that transports $E$ to $F$ (i.e., $\Phi_{\#}(\Leb^d\measurerestr E)=\Leb^d\measurerestr F$).

\begin{lemma}\label{lem: prop_Wp}
For any bounded Lebesgue measurable set $E\subseteq\R^d$ and $p\geqslant 1$, let $F$ denote a $\W_p$-minimizer of $E$ and $\Phi$ denote an optimal transport map that transports $E$ to $F$. Then there is a constant $C_0(d)=(3^{1/d}+2)\ell_d$ such that
\begin{enumerate}[label=(\alph*)]
    \item For a.e. $x\in E$
    \begin{equation}\label{eq: transport_dis}
        \abs{\Phi(x)-x}\leqslant C_0(d)\abs{E}^{1/d}.
    \end{equation}
    \item 
    \begin{equation}\label{eq: Wasserstein_uni_bd}
        \W_p(E)\leqslant C_0(d)\abs{E}^{\frac{1}{p}+\frac{1}{d}}.
    \end{equation}
    \item
    \begin{equation}\label{eq: F_bnd}
    \abs{F\setminus \left\{y\in\R^d: \textrm{dist}(y,E)\leqslant C_0(d)\abs{E}^{1/d}\right\}}=0.
    \end{equation}
\end{enumerate}
\end{lemma}
\begin{proof}
Without loss of generality, by \eqref{eq: W_rescale} we may assume that $\abs{E}=1$.

Let $K=\left\{x\in E: \abs{\Phi(x)-x}>(3^{1/d}+2)\ell_d\right\}$. We want to show that $\abs{K}=0$. Indeed, assume $\abs{K}>0$, then there exists $x_0\in K$ such that for some $0<r\leqslant \ell_d$, we have $\abs{K\cap B(x_0,r)}>0$.

Since $\abs{K\cap B(x_0,r)}\leqslant \abs{B(x_0,r)}\leqslant \abs{B(x_0,\ell_d)}=1$ and
\[\abs{B(x_0,3^{1/d}\ell_0)\setminus (E\cup F)}\geqslant\abs{B(x_0,3^{1/d}\ell_d)}-\abs{E}-\abs{F}=3-1-1=1,\]
there exists a subset $H\subseteq B(x_0,3^{1/d}\ell_d)\setminus (E\cup F)$ with $\abs{H}=\abs{K\cap B(x_0,r)}$ and $\abs{H \cap (K\cap B(x_0,r))}=0$. Let $\Psi$ be an optimal transport map from $K\cap B(x_0,r)$ to $H$.

Now we construct a new mapping:

    \begin{equation*}
    \bar{\Phi}(x)=\left\{
    \begin{aligned}
    &\Phi(x)\qquad x\in E\setminus (K\cap B(x_0,r));\\
    &\Psi(x)\qquad x\in K\cap B(x_0,r).
    \end{aligned}
    \right.
    \end{equation*}
    
By our construction, $\abs{\bar{\Phi}(E)\cap E}=0$. Note that for 
\[\abs{\bar{\Phi}(E)}=\abs{\Psi(K\cap B(x_0,r))}+\abs{\Phi(E\setminus (K\cap B(x_0,r)))}=\abs{K\cap B(x_0,r)}+\abs{E\setminus (K\cap B(x_0,r))}=\abs{E},\]
Thus $(E,\bar{\Phi}(E))\in \mathcal{F}_1$. Moreover, for a.e. $x\in K\cap B(x_0,r)$,
\begin{align*}
        \abs{\Psi(x)-x}&\leqslant\abs{\Psi(x)-x_0}+\abs{x_0-x}\\
        &\leqslant 3^{1/d}\ell_d+r\leqslant (3^{1/d}+1)\ell_d\\
        &<\abs{\Phi(x)-x}-\ell_d.
\end{align*}

Thus, since $\abs{K\cap B(x_0,r)}>0$, it holds that
\[\int_E \abs{\bar{\Phi}(x)-x}^p\diff{x}-\int_E\abs{\Phi(x)-x}^p\diff{x}=\int_{K\cap B(x_0,r)} \abs{\Psi(x)-x}^p\diff{x}-\int_{K\cap B(x_0,r)}\abs{\Phi(x)-x}^p\diff{x}<0.\]
This shows that 
\[W_p(E,\bar{\Phi}(E))<W_p(E,F),\]
a contradiction with $F$ being the $\W_p$-minimizer of $E$.

Hence for a.e. $x\in E$,
    \[\abs{\Phi(x)-x}\leqslant C_0(d)\abs{E}^{1/d}.\]
As a result,
\[\W_p(E)=W_p(E,F)=\left(\int_E\abs{\Phi(x)-x}^p\diff{x}\right)^{1/p}\leqslant C_0(d)\abs{E}^{\frac{1}{p}+\frac{1}{d}},\]
and 
\[\abs{F\setminus \left\{y\in \R^d : \textrm{dist}(y,E)\leqslant C_0(d)\abs{E}^{1/d}  \right \}}=0.\]

\end{proof}

\begin{lemma}[Lower semi-continuity of $\W_p$]\label{lem: lsc_Wp}
Suppose $\{E_n\}$ is any sequence of sets of finite perimeter in $\R^d$ with 
\[\sup_{n}P(E_n)<\infty\qquad\mathrm{and}\qquad E_n\subseteq B_R\]
for each $n$ and some $R>0$. If $E_n$ converges to $E$, then we have
\[\W_p(E)\leqslant \liminf_{n\to\infty}\W_p(E_n).\]
\end{lemma}
\begin{proof}
By the definition of $\liminf$, up to extracting a subsequence of $\left\{E_n\right\}$ if necessary (still denoted by $\left\{E_n\right\}$), we may assume that 
\[\lim_{n\to\infty}\W_p(E_n)=\liminf_{n\to\infty}W_p(E_n).\]

Let $F_n$ denote corresponding $\W_p$-minimizer of $E_n$ such that $\W_p(E_n)=W_p(E_n,F_n)$. By Theorem 3.13 and Remark 3.14 in \cite{buttazzo2017wasserstein}, $F_n$ is also a set of finite perimeter with a uniform bound on its perimeter. Furthermore, by \eqref{eq: F_bnd} in \prettyref{lem: prop_Wp}, $\left\{F_n\right\}$ are contained in $B_{R'}$ for $R'=R+C_0(d)\abs{E}^{1/d}$. Thanks to the compactness of sets of finite perimeter, there exists a set $F$ of finite perimeter in $B_{R'}$ and a subsequence $\{F_{n_k}\}$ such that $F_{n_k}\to F$.

Since $W_p$ is lower semi-continuous with respect to weak convergence, we have
\[W_p(E,F)\leqslant \liminf_{n\to\infty} W_p(E_{n_k},F_{n_k}).\]
For any $k$,
\[E\cap F\subseteq (E\setminus E_{n_k})\cup (F\setminus F_{n_k}) \cup (E_{n_k}\cap F_{n_k}),\]
which yields that $\abs{E\cap F}=0$. Therefore,
\[\W_p(E)\leqslant W_p(E,F)\leqslant \liminf_{k\to\infty} W_p(E_{n_k},F_{n_k})=\liminf_{k\to\infty}\W_p(E_{n_k})=\lim_{n\to\infty}\W_p(E_n)=\liminf_{n\to\infty}W_p(E_n).\]

\end{proof}

\section{Existence of minimizing sequence of bounded sets}\label{sec: bnd_seq}
In this section, we will prove the following theorem:
\begin{theorem}\label{thm: alter_bound_seq}
There exists a minimizing sequence of bounded sets to problem (\ref{problem: m}).
\end{theorem}
We will show that for any minimizing sequence $(E_n, F_n)$ to (\ref{problem: m}), there is an alternative minimizing sequence $(\widetilde{E_n}, \widetilde{F_n})$ of bounded sets to (\ref{problem: m}).
\begin{remark}
Here, $(\widetilde{E_n}, \widetilde{F_n})$ is not necessarily uniformly bounded.
\end{remark}

We first start with an important lemma, originating from Almgren's breakthrough work in \cite{Almgren1976Existence}, and rephrased in \cite{maggi2012sets}:

\begin{lemma}[Nucleation, \cite{maggi2012sets}]
\label{lem: nucleation}
For every $d\geqslant 2$, there exists a positive constant $c(d)$ with the following property: given any set $E\subseteq \R^d$ of finite perimeter with $0<\abs{E}<\infty$, and any positive number $\varepsilon$ with
$\varepsilon\leqslant \min\{\abs{E}, \frac{P(E)}{2dc(d)}\}$, there exists a finite family of points $I\subseteq \R^d$ such that:
\[
\abs{E\setminus \bigcup_{x\in I}B(x,2)}<\varepsilon\quad\textrm{and}\quad
\abs{E\cap B(x,1)}\geqslant \left(c(d)\frac{\varepsilon}{P(E)}\right)^{d},\quad \forall x\in I.
\]

Moreover, $\abs{x-y}>2$ for every $x,y\in I, x\neq y$, and 
$$\# I \leqslant \abs{E}\left(\frac{P(E)}{c(d)\varepsilon}\right)^d.$$
\end{lemma}

Using this lemma, we have the following proposition:
\begin{proposition}\label{prop: perimeter_inc}
Let $E\subseteq\R^d$ be a set of finite perimeter with $\abs{E}<\infty$ and $d\geqslant 2$. For any number $0<\varepsilon\leqslant \min\{\abs{E}, \frac{P(E)}{2dc(d)}\}$, there exists a finite subset $I\subseteq \R^d$ with 
\[\# I \leqslant \abs{E}\left(\frac{P(E)}{c(d)\varepsilon}\right)^d\] such that for some number $r\in [2,3]$, the set 
\[U:=\bigcup_{x\in I}B(x,r)\]
satisfies
\[
\abs{E\setminus U}<\varepsilon \quad\textrm{and}\quad \mathcal{H}^{d-1}(E\cap \partial U)\leqslant \varepsilon.    
\]
\end{proposition}
\begin{proof}
By \prettyref{lem: nucleation}, there exists a finite set $I\subseteq \R^d$ such that:
$$\abs{E\setminus \bigcup_{x\in I}B(x,2)}< \varepsilon\qquad\textrm{and}\qquad\qquad \# I \leqslant \abs{E}\left(\frac{P(E)}{c(d)\varepsilon}\right)^d.$$

We now consider the function $f:\R^d\rightarrow \R$ defined by
\[f(y):=\min_{x\in I} \abs{y-x}, \]
which gives the distance from the point $y$ to the finite set $I$. It is a Lipschitz function with $\abs{\nabla f(y)}=1$ for a.e. $y$ in $\R^d$. Using this function, we see that
\[A:=E\setminus \bigcup_{x\in I}B(x,2)=E\cap f^{-1}([2,\infty)).\]

According to the coarea formula (see Theorem 1 in Section 3.4.2 in \cite{Evans1992Measure}): 
\[
\int_A \abs{\nabla f(y)} d\Leb^d(y)=\int_{\R}\mathcal{H}^{d-1}(A\cap f^{-1}(t))\diff{t}.
\]
That is,
\[
\abs{A}=\int_{2}^{\infty}\mathcal{H}^{d-1}(E\cap f^{-1}(t))\diff{t}.
\]
Since $\abs{A}<\varepsilon$, in particular it follows that
\[
\int_{2}^{3}\mathcal{H}^{d-1}(E\cap f^{-1}(t))\diff{t}\leqslant\abs{A}<\varepsilon.
\]
As a result, there exists a $r\in [2,3]$ such that
\[\mathcal{H}^{d-1}(E\cap f^{-1}(r))<\epsilon.\]
Now, for the set $U:=\bigcup_{x\in I}B(x,r)$, it holds that
\[
\abs{E\setminus U}\leqslant \abs{E\setminus \bigcup_{x\in I}B(x,2)}<\varepsilon\quad\textrm{and}\quad \mathcal{H}^{d-1}(E\cap \partial U)=\mathcal{H}^{d-1}(E\cap f^{-1}(r))\leqslant \varepsilon.    
\]
\end{proof}

\begin{theorem}\label{thm: rearrange}
For any $m>0$, $(E,F)\in \mathcal{F}_m$, and 
\begin{equation}\label{eqn: epsilon_bound}
    0<\varepsilon\leqslant \min\left\{\abs{E},~  \frac{P(E)}{2dc(d)}\right\},
\end{equation}
there exists  $(\widetilde{E},\widetilde{F})\in \mathcal{F}_m$ such that
\[ P(\widetilde{E})\leqslant P(E)+2\varepsilon, \qquad W_p(\widetilde{E},\widetilde{F})\leqslant W_p(E,F)+ \left(\frac{2}{\omega_d}\right)^{1/d} \varepsilon ^{\frac{1}{p}+\frac{1}{d}},
\]
and $(\widetilde{E},\widetilde{F})\in \mathcal{F}$ are bounded sets inside the ball $B(O, R_\varepsilon)$ 
where $O=(0,\cdots,0)$ is the origin in $\R^d$,
\[R_\varepsilon:=\left(6 \left(\frac{P(E)}{c(d)\varepsilon}\right)^d+C_0(d) \left(\frac{P(E)}{c(d)\varepsilon}\right)^{d-1}\right) \abs{E}+\left(\frac{2\varepsilon}{\omega_d}\right)^{1/d}.\]

\end{theorem}

\begin{proof}
By Proposition \ref{prop: perimeter_inc},
 there exists a finite subset $I$ in $\R^d$ and a positive constant $r\in [2,3]$ such that the set
 \[U:=\bigcup_{x\in I} B(x,r)\] satisfies
\[
\abs{E\setminus U}<\varepsilon,\qquad \mathcal{H}^{d-1}(E\cap \partial U)\leqslant \varepsilon \quad\textrm{and}\quad \# I \leqslant \abs{E}\left(\frac{P(E)}{c(d)\varepsilon}\right)^d.\]

\begin{figure}[!hb]
    \centering
    \includegraphics[width=\columnwidth]{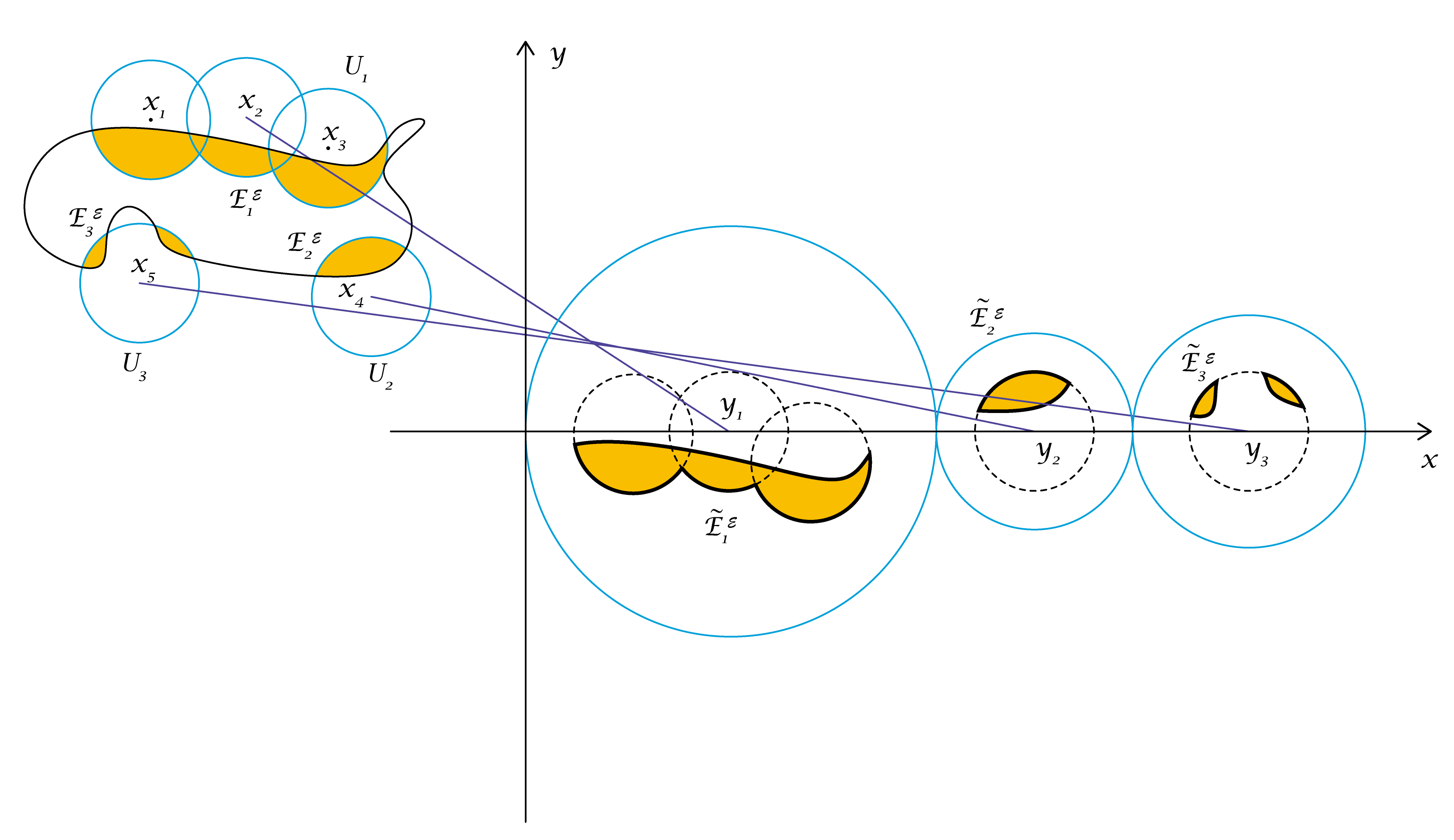}
    \caption{We use balls of fixed radius $r$ to cover the majority of $E$. For each connected part $E_j^{\varepsilon}$ combined with $F_j^{\varepsilon}$, we pack each pair $(E_j^{\varepsilon},F_j^{\varepsilon})$ into a ball and then align these balls together inside $B(O,R_{\varepsilon})$. For simplicity and clearness, we do not demonstrate corresponding parts from $F$.}
    \label{fig: WWDesign}
\end{figure}

Denote 
\[E^\varepsilon \defeq E\cap U\quad \textrm{and}\quad E_0^{\varepsilon}\defeq E \setminus U.\]
Then, $E=E^\varepsilon \cup E_0^{\varepsilon}$,
$$\abs{E_0^{\varepsilon}}= \abs{E\setminus \bigcup_{x\in I} B(x,r)}<\varepsilon,\quad\textrm{and}\quad
\abs{E^{\varepsilon}}=\abs{E}-\abs{E_0^{\varepsilon}}>\abs{E}-\varepsilon.$$

Since $\# I<\infty$, there are at most $\# I$ many connected components of $U$. Let $I=\bigsqcup_{j=1}^{k} I_j$, where $\left\{I_1,I_2,\cdots,I_k\right\}$ is a partition of $I$ and $k\leqslant \#I$, such that for each $j=1,2,\cdots, k$, 
$$U_j:=\bigcup_{x\in I_j} B(x,r)$$
is a connected component of $U$.
For each $j=1,2,\cdots,k$, denote 
\[E_j^\varepsilon \defeq E\cap U_j.\]
Then $E^{\varepsilon}=\bigcup_{j=1}^{k} E^{\varepsilon}_j$, and $E_j^{\varepsilon} \subseteq U_j \subseteq B(x_j, 2rn_j)$ for some point $x_j\in I_j$, where $n_j=\# I_j$ denotes the number of points in $I_j$.

Note that
\begin{equation}\label{eqn: perimeter_sum}
\sum_{j=0}^kP(E_j^\varepsilon)
=P(E\setminus U)+P(E\cap U)
=P(E)+2\H^{d-1}(E\cap \partial U\})
\leqslant P(E)+2\varepsilon.
\end{equation}

Since $\abs{F}=\abs{E}$ and $\abs{F\cap E}=0$, there exists an optimal transport map $\Phi$ that transports $E$ to $F$.
For $j=1,2,\cdots,k$,  let $F_j^{\varepsilon}:=\Phi(E_j^{\varepsilon})$ be the image of $E_j^{\varepsilon}$ under $\Phi$. Then
\begin{eqnarray} \label{eqn: upper_bound_on_W_sums}
\sum_{j=1}^k W_p^p( E_j^{\varepsilon}, F_j^{\varepsilon})=\sum_{j=1}^k \int_{E_j^{\varepsilon}} \abs{x-\Phi(x)}^p \diff{\Leb^d(x)} = \int_{E^{\varepsilon}} \abs{x-\Phi(x)}^p \diff{\Leb^d(x)} \leqslant W_p^p(E,F)
.
\end{eqnarray}

Now, let $\hat{F}_j^{\varepsilon}$ be a $\W_p$-minimizer of the bounded set $E^\varepsilon_j$. Then, 
$\abs{\hat{F}_j^{\varepsilon}}=\abs{E_j^{\varepsilon}}$, $\abs{\hat{F}_j^{\varepsilon}\cap E_j^{\varepsilon}}=0$ and 
\[W_p(E_j^{\varepsilon}, \hat{F}_j^{\varepsilon})\leqslant W_p( E_j^{\varepsilon}, F_j^{\varepsilon}).\]

Since $E^\varepsilon_j \subseteq B(x_j, 2rn_j)\subseteq B(x_j, 6n_j)$, by \eqref{eq: F_bnd} in \prettyref{lem: prop_Wp}, it follows that
\[\hat{F}_j^{\varepsilon}\subseteq B\left(x_j, 6n_j+C_0(d)\abs{E^\varepsilon_j}^{1/d}\right).\]

Note that
\begin{align*}
&\sum_{j=1}^k \textrm{diam}\left(B\left(x_j, 6n_j+C_0(d)\abs{E^\varepsilon_j}^{1/d}\right)\right)\\
=&\sum_{j=1}^k \left(12n_j+2C_0(d)\abs{E^\varepsilon_j}^{1/d}\right)=
12\cdot \#I+2C_0(d)\sum_{j=1}^k\abs{E^\varepsilon_j}^{1/d} \\
\leqslant& 12\cdot \#I+2C_0(d)k^{1-\frac{1}{d}}\left(\sum_{j=1}^k \abs{E^\varepsilon_j}\right)^{1/d}
\leqslant
12 \cdot \# I+2C_0(d) (\# I) ^{1-\frac{1}{d}} \abs{E}^{1/d}\\
\leqslant&
12\abs{E}\left(\frac{P(E)}{c(d)\varepsilon}\right)^d+2C_0(d) \left(\abs{E}\left(\frac{P(E)}{c(d)\varepsilon}\right)^d \right) ^{1-\frac{1}{d}} \abs{E}^{1/d}\\
=&
12 \abs{E}\left(\frac{P(E)}{c(d)\varepsilon}\right)^d+2C_0(d) \left(\frac{P(E)}{c(d)\varepsilon}\right)^{d-1} \abs{E}=2R_\varepsilon-2\rho_\varepsilon,
\end{align*}
where 
\[\rho_\varepsilon=\left(\frac{2\varepsilon}{\omega_d}\right)^{1/d}.\]
Thus, inside the ball $B(O, R_\varepsilon)$, one may pick $k+1$ pairwise disjoint closed balls \[\{\overline{B(y_j, 6n_j+C_0(d)\lvert E^\varepsilon_j \rvert^{1/d})}\}_{j=1}^k\cup \overline{B(y_0, \rho_\varepsilon)}.\]

For each $j=1,\cdots, k$, define
\[\widetilde{E}_j^{\varepsilon}=E_j^{\varepsilon}+(y_j-x_j)\quad \textrm{and}\quad \widetilde{F}_j^{\varepsilon}=\hat{F}_j^{\varepsilon}+(y_j-x_j).\]
i.e., we translate the pair $(E_j^{\varepsilon}, \hat{F}_j^{\varepsilon})$ in the ball $B(x_j,6n_j+C_0(d)\abs{E^\varepsilon_j}^{1/d})$ to the corresponding pair $(\widetilde{E}_j^{\varepsilon},\widetilde{F}_j^{\varepsilon})$ inside the ball $B(y_j,6n_j+C_0(d)|E^\varepsilon_j|^{1/d})$, as shown see Figure \ref{fig: WWDesign}.

Since both the perimeter and the Wasserstein distance are translation invariant, we have \[P(\widetilde{E}_j^{\varepsilon})=P(E_j^{\varepsilon})\quad \textrm{and}\quad
W_p(\widetilde{E}_j^{\varepsilon},\widetilde{F}_j^{\varepsilon})=W_p(E_j^{\varepsilon},\hat{F}_j^{\varepsilon}).\]
Also denote 
\[\widetilde{F}_0^{\varepsilon}\defeq B(y_0, t_\varepsilon)\setminus B(y_0, s_\varepsilon)\quad\textrm{and}\quad \widetilde{E}_0^{\varepsilon}\defeq B(y_0, s_\varepsilon),\]
with
\[t_\varepsilon=\left(\frac{2\abs{E_0^\varepsilon}}{\omega_d}\right)^{1/d}\quad\textrm{and}\quad s_\varepsilon=\left(\frac{\abs{E_0^\varepsilon}}{\omega_d}\right)^{1/d}.\]
Note that 
\[\abs{\widetilde{E}_0^{\varepsilon}}=\abs{E_0^\varepsilon}=\abs{\widetilde{F}_0^{\varepsilon}}.\]
Since $\abs{E_0^\varepsilon}\leqslant \varepsilon$, it follows that $0<s_\varepsilon<t_\varepsilon\leqslant \rho_\varepsilon$. Therefore, both sets $\widetilde{E}_0^{\varepsilon}$ and $\widetilde{F}_0^{\varepsilon}$ are contained in $B(y_0, \rho_\varepsilon)$.
Now, define
\[\widetilde{E}\defeq\bigcup_{j=0}^k\widetilde{E}_j^\varepsilon\quad\textrm{and}\quad\widetilde{F}\defeq\bigcup_{j=0}^k\widetilde{F}_j^\varepsilon.\]

Then 
\begin{align*}
    \abs{\widetilde{E}\cap \widetilde{F}}&=\abs{\bigcup_{j=0}^k\widetilde{E}_j^\varepsilon \cap \bigcup_{j=0}^k\widetilde{F}_j^\varepsilon}=\sum_{j=0}^k \abs{\widetilde{E}_j^\varepsilon\cap \widetilde{F}_j^\varepsilon}=0,\\
    \abs{\widetilde{E}}&=\sum_{j=0}^k \abs{\widetilde{E}_j^\varepsilon}= \sum_{j=0}^k \abs{E_j^\varepsilon}=\abs{E},
\end{align*}
and similarly $\abs{\widetilde{F}}=\abs{F}$. As a result, $(\widetilde{E}, \widetilde{F})\in \mathcal{F}$.

Moreover, by applying the isoperimetric inequality on $\widetilde{E}_0^{\varepsilon}$, \eqref{eqn: perimeter_sum} implies that
\begin{align*}
P(\widetilde{E})=\sum_{j=0}^kP(\widetilde{E}_j^\varepsilon)
\leqslant \sum_{j=0}^kP({E}_j^\varepsilon)\leqslant P(E)+2\varepsilon.
\end{align*}

Furthermore,
\begin{align*}
W_p^p(\widetilde{E},\widetilde{F})
&\leqslant\sum_{j=0}^k W_p^p(\widetilde{E}_j^\varepsilon, \widetilde{F}_j^\varepsilon)\\
&=W_p^p(\widetilde{E}_0^\varepsilon, \widetilde{F}_0^\varepsilon)+\sum_{j=1}^k W_p^p( E_j^\varepsilon, F_j^\varepsilon)\\
&\leqslant(t_\varepsilon)^p \abs{E_0^\varepsilon}+\sum_{j=1}^k W_p^p( E_j^\varepsilon, F_j^\varepsilon)\\
&\leqslant\left(\frac{2\varepsilon}{\omega_d}\right)^{p/d} \varepsilon+W_p^p(E,F),
\end{align*}
by \eqref{eqn: upper_bound_on_W_sums}.
Thus, since $p\geqslant 1$, it follows that
\[W_p(\widetilde{E},\widetilde{F})\leqslant \left( \left(\frac{2\varepsilon}{\omega_d}\right)^{p/d} \varepsilon+W_p^p(E,F) \right)^{1/p} \leqslant \left(\frac{2\varepsilon}{\omega_d}\right)^{1/d}\varepsilon^{1/p}+W_p(E,F).\]
\end{proof}

Now we use \prettyref{thm: rearrange} to prove \prettyref{thm: alter_bound_seq}.

\begin{proof}[Proof of \prettyref{thm: alter_bound_seq}]
Let $\left\{(E_n, F_n)\right\}$ be any minimizing sequence of the functional $P(E)+W_p(E,F)$ in $ \mathcal{F}_m$.
For each $n$, pick $\varepsilon_n$ small enough  so that
\[
 0<\varepsilon_n\leqslant \min\left\{\frac{1}{n}, m,~  \frac{P(E_n)}{2dc(d)}\right\}.
 \]
 By \prettyref{thm: rearrange},  there exist bounded sets $(\widetilde{E_n},\widetilde{F_n})\in \mathcal{F}_m$ contained in the ball $B(O,R_{\varepsilon_n})$, such that 
\[ P(\widetilde{E_n})\leqslant P(E_n)+\frac{2}{n}, \quad\textrm{and}\quad W_p(\widetilde{E_n},\widetilde{F_n})\leqslant W_p( E_n, F_n)+ \left(\frac{2}{\omega_d}\right)^{1/d} \left(\frac{1}{n}\right) ^{\frac{1}{p}+\frac{1}{d}}.
\]
Thus,
\[P(\widetilde{E_n})+W_p(\widetilde{E_n},\widetilde{F_n})\leqslant P(E_n)+W_p( E_n, F_n)+\frac{2}{n}+ \left(\frac{2}{\omega_d}\right)^{1/d} \left(\frac{1}{n}\right) ^{\frac{1}{p}+\frac{1}{d}}.\]
This shows the sequence of the bounded sets $\left\{(\widetilde{E_n},\widetilde{F_n})\right\}$ is also a minimizing sequence of the functional $P(E)+W_p(E,F)$ in $ \mathcal{F}_m$.
\end{proof}

\begin{corollary}
For any $m\geqslant 0$, the minimizing problem (\ref{problem: m}) is equivalent to the problem
\begin{equation}\label{problem: bounded} 
 \text{ Minimize }\qquad   P(E)+ W_p(E, F)
\qquad\text{ among all bounded sets } (E, F)\in \mathcal{F}_m.
\end{equation}
\end{corollary}

To solve problem (\ref{problem: bounded}), it is sufficient to solve the problem
\begin{equation}\label{problem: T} 
 \text{ Minimize }\qquad   T(E):=P(E)+ \W_p(E)
 \end{equation}
 among all bounded set $E\subseteq \R^d$ of finite perimeter with $\abs{E}=m$. Any solution $E^*$ of problem (\ref{problem: T}) together with its $\W_p$-minimizer $F^*$ provides a solution $(E^*, F^*)$ to problem (\ref{problem: bounded}), and vice versa.

\section{Existence of minimizers for small volume}\label{sec: small_vol}
In this section, we will show that problem (\ref{problem: T}) has a solution when $m$ is small and $\frac{1}{p}+\frac{2}{d}>1$. Our work is inspired by \cite{Knupfer2014Isoperimetric} as mentioned in the introduction.

\begin{theorem}\label{thm: minTset}
Suppose $d\geqslant 1, p\geqslant 1 $ with $\frac{1}{p}+\frac{2}{d}>1$, there exists an $m_0>0$ such that for any $m\leqslant m_0$, the minimization problem (\ref{problem: T}) has a solution.
\end{theorem}

Recall \prettyref{thm: minTset} is equivalent to \prettyref{thm: minEset} and \prettyref{thm: minGset}.

To do so, we start with a few technique propositions.

\begin{proposition}[Nonoptimality]\label{prop: nonoptimality1}
Suppose $d\geqslant 1, p\geqslant 1 $ with $\frac{1}{p}+\frac{2}{d}>1$,
let $G\subseteq \R^d$ be a bounded set of finite perimeter with $\abs{G}=m< \min\{1,\omega_d\}$. Suppose there is a partition of $G$ into two disjoint sets of finite perimeter $G_1$ and $G_2$ with positive volumes such that
    \begin{align}\label{eq: cond_nonopt}
        P(G_1)+P(G_2)-P(G)\leqslant \frac{1}{2}T(G_2).
    \end{align}
    Then there is an $\varepsilon=\varepsilon(m,d)>0$ such that if
    \[\abs{G_2}\leqslant \varepsilon \abs{G_1},\]
    there exists a bounded set $E\subseteq \R^d$ such that $\abs{E}=\abs{G}$ and $T(E)<T(G)$. 
\end{proposition}
\input{separate.tikz}
\begin{proof}
Let $r:=\left(\frac{m}{\omega_d}\right)^{1/d}<1$. Then $\abs{B_r}=\omega_d r^d=m=\abs{G}$. 
When $T(B_r)<T(G)$, the result holds for $E=B_r$. Thus, without loss of generality, we may assume that $T(B_r)\geqslant T(G)$.

By direct computation and \eqref{eq: Wasserstein_uni_bd} in \prettyref{lem: prop_Wp}, 
\[P(B_r)=d\omega_d r^{d-1}=d\omega_d^{1/d} m^{1-1/d}\quad\textrm{and}\quad \mathcal{W}_p(B_r)\leqslant C_0(d) m^{\frac{1}{p}+\frac{1}{d}},\]
thus,
\begin{equation}\label{eq: T2_upperbound}
    T(G)\leqslant T(B_r)\leqslant C(d)\max{(m^{1/p+1/d},m^{1-1/d})}\leqslant C(d)m^{1-1/d},
\end{equation}
because $\frac{1}{p}+\frac{1}{d}>1-\frac{1}{d}$ and $m<1$.
Since $G$ is the disjoint union of $G_1$ and $G_2$, it follows that
\begin{equation}\label{eq: W_G_union}
    \W_p(G_1)+\W_p(G_2)- \W_p(G)\leqslant 0.
\end{equation}
Together with (\ref{eq: cond_nonopt}), it follows that
\[
T(G_1)+T(G_2)-T(G) \leqslant\frac{1}{2}T(G_2),
\]
i.e.
\begin{equation}\label{eq: T_inequality}
    T(G_1)+\frac{1}{2}T(G_2)\leqslant T(G).
\end{equation}

Let 
\[
\gamma=\frac{\abs{G_2}}{\abs{G_1}}, \ \ell=(1+\gamma)^{1/d},\quad\textrm{and}\quad E=\ell G_1.
\]
Note that when $\frac{1}{p}+\frac{2}{d}>1$, it follows that $d-1<1+\frac{d}{p}$. Note that $\ell>1$, thus
\[
        T(E)=\ell^{d-1}P(G_1)+\ell^{1+d/p}\W_p(G_1)
        \leqslant \ell^{1+d/p}(P(G_1)+\W_p(G_1))=\ell^{1+d/p}T(G_1).
\]
As a result,
    \begin{align*}
        T(E)-T(G)&\leqslant T(G_1)-T(G)+(\ell^{1+d/p}-1) T(G_1)\\
        &=T(G_1)+T(G_2)-T(G) -T(G_2)+(\ell^{1+d/p}-1)T(G_1)\\
        &=\left(P(G_1)+P(G_2)-P(G)\right) + \left(\mathcal{W}_p(G_1)+\mathcal{W}_p(G_2)-\mathcal{W}_p(G)\right) -T(G_2)+(\ell^{1+d/p}-1) T(G_1)\\ 
        &\stackrel{\eqref{eq: cond_nonopt}}{\leqslant}\frac{1}{2}T(G_2)+0-T(G_2)+(\ell^{1+d/p}-1)T(G_1),\\
        &=-\frac{1}{2}T(G_2)+(\ell^{1+d/p}-1)T(G_1)\\
        &= -\frac{1}{2}T(G_2)+((1+\gamma)^{1/d+1/p}-1)T(G_1)\\
         &\leqslant -\frac{1}{2}T(G_2)+2(1/d+1/p)\gamma T(G_1)
    \end{align*}
    when $\gamma>0$ is small enough.
    By isoperimetric inequality, we have 
    \[T(G_2)\geqslant P(G_2)\geqslant C(d)\abs{G_2}^{1-1/d}.\]
    On the other hand, by \eqref{eq: T_inequality} and \eqref{eq: T2_upperbound}, when $\gamma<1/2$,
    \[
   \gamma T(G_1)\leqslant \gamma  T(G)\leqslant C(d)\gamma m^{1-1/d}=C(d)\gamma^{1/d}(\gamma m)^{1-1/d}\leqslant C(d)\gamma^{1/d}(2\abs{G_2})^{1-1/d}.
    \]
   Hence,  combine those inequalities, we have
    \begin{align*}
        T(E)-T(G)&\leqslant -C(d)\abs{G_2}^{1-1/d}+C(d)\gamma^{1/d} \abs{G_2}^{1-1/d}\\
        &\leqslant -C(d)\abs{G_2}^{1-1/d}+C(d)\varepsilon^{1/d} \abs{G_2}^{1-1/d}\\
        &<0,
     \end{align*}
    for $\varepsilon$ sufficiently small.
\end{proof}

%\section{Existence}\label{sec: Existence}
Using the above proposition, we have the following uniform boundedness result:
\begin{theorem}\label{thm: comp_bdset}
Suppose $p\geqslant$, $d\geqslant1$ with $\frac{1}{p}+\frac{2}{d}>1$, there exists an $m_0>0$ such that for every bounded set $G\subseteq \R^d$ of finite perimeter with $\abs{G}\leqslant m_0$, there exists a bounded set $E\subseteq \R^d$ of finite perimeter with 
\begin{align}
\label{eq: comp_bdset}
   \abs{E}=\abs{G}, \qquad  T(E)\leqslant T(G)\quad\textrm{and}\quad E\subseteq B_2.
\end{align}
\end{theorem}
\begin{proof}
We may assume $m:=\abs{G}\leqslant \min\{1,\omega_d\}$, and set $r:=\left(\frac{m}{\omega_d}\right)^{1/d}$. Note that $r\leqslant 1$ and $\abs{B_r}=\omega_d r^d=\abs{G}$. Thus, when $T(G)\geqslant T(B_r)$, the set $E=B_r$ satisfies \eqref{eq: comp_bdset}. As a result, without loss of generality, we may assume that 
\[T(G)< T(B_r).\]
That is,
\[P(G)+\W_p(G)<P(B_r)+\W_p(B_r).\]
%We may also assume that $T_2(F)\leqslant T_2(B_r)$. Otherwise we can pick $G=B_r$ and \eqref{eq: comp_bdset} hold.
Thus, we have the following upper bound for the isoperimetric deficit of $G$: 
\begin{align*}
    D(G)&=\frac{P(G)-P(B_r)}{P(B_r)}<\frac{\W_p(B_r)-\W_p(G)}{P(B_r)}\\
    &\leqslant\frac{\W_p(B_r)}{P(B_r)}\leqslant\frac{C_0(d)\abs{B_r}^{\frac{1}{p}+\frac{1}{d}}}{P(B_r)}\\
    &=\frac{C_0(d)(\omega_d r^d)^{\frac{1}{p}+\frac{1}{d}}}{d\omega_d r^{d-1}}=\frac{C_0(d)}{d(\omega_d )^{1-\frac{1}{p}-\frac{1}{d}}}r^{\alpha},
\end{align*}
where $\alpha:=2+d(\frac{1}{p}-1)>0$.
%Thanks to \eqref{eq: figalli_isop}, we have $\Delta(F,B_r)\leqslant C(d,p)r^{\alpha/2}$.
By \eqref{eq: figalli_isop}, and up to a suitable translation, we have 
\begin{align}
\label{eq: volume_rate}
\frac{\abs{G\setminus B_r}}{\abs{G}}&\leqslant
\frac{\abs{G\Delta B_r}}{\abs{G}}= \Delta(G,B_r)\leqslant C(d)\sqrt{D(G)}\nonumber\\
&\leqslant C(d)\sqrt{\frac{C_1(d,p)}{d(\omega_d )^{1-\frac{1}{p}-\frac{1}{d}}}r^{\alpha}}=C(d,p)r^{\alpha/2},
\end{align}
where
\[C(d,p)=C(d)\sqrt{\frac{C_1(d,p)}{d(\omega_d )^{1-\frac{1}{p}-\frac{1}{d}}}}. \]

Let $m_0>0$ be small enough such that
\[C(d,p)(\frac{m_0}{\omega_d})^{\alpha/2}<1.\]
Then,
\[C(d,p)r^{\alpha/2}=C(d,p)(\frac{m}{\omega_d})^{\alpha/2}<1.\]

Since the function $\frac{x}{1-x}$ is increasing on $[0,1)$, we have when $r>0$ is small enough,
\[\frac{\abs{G\setminus B_r}/\abs{G}}{1-\abs{G\setminus B_r}/\abs{G}}\leqslant \frac{C(d,p)r^{\alpha/2}}{1-C(d,p)r^{\alpha/2}} \leqslant \varepsilon,\]
where $\varepsilon$ is given in Proposition \ref{prop: nonoptimality1}.
That is, 
 %For sufficiently small $m$, or equivalently sufficiently small $r$, there exists $\varepsilon=\varepsilon(m,d)$ such that
\[\abs{G\setminus B_r}\leqslant \varepsilon\abs{G\cap B_r},\]
 for $\abs{G}=\abs{G\setminus B_r}+\abs{G\cap B_r}$. Note that for all $t\geqslant r$, it also follows that
\[\abs{G\setminus B_t}\leqslant \varepsilon\abs{G\cap B_t}.\]

Case 1: When $P(G\cap B_t)+P(G\setminus B_t)-P(G)\leqslant \frac{1}{2}T(G\setminus B_t)$ for some $t\in [r,1]$, by Proposition \ref{prop: nonoptimality1}
there exists a bounded set $E$ with $T(E)\leqslant T(G)$. By the proof of Proposition \ref{prop: nonoptimality1}, either $E=B_r \subseteq B_2$ or $E=\ell (G\cap B_t)\subseteq \ell B_t\subseteq B_2$, where $\ell \leqslant (1+\varepsilon)^{1/d} \leqslant 2 $.

Case 2: When $P(G\cap B_t)+P(G\setminus B_t)-P(G)\geqslant \frac{1}{2}T(G\setminus B_t)$ for all $t\in [r,1]$, we have the following observations. By the coarea formula (see Proposition 1 in Section 3.4.4 in \cite{Evans1992Measure}), for almost every $t\in [r, 1]$, 
\begin{align*}
\frac{\mathrm{d}}{\diff{t}}\abs{G\cap B_t} &=\frac{\mathrm{d}}{\diff{t}}\left(\int_{B_t}\Id_G \diff{x} \right)\\&=\H^{d-1}(G\cap \partial B_t)\\
    &=\frac{1}{2}(P(G\cap B_t)+P(G\setminus B_t)-P(G))\\
    &\geqslant \frac{1}{2}\cdot\frac{1}{2}T(G\setminus B_t)\\
    &\geqslant \frac{1}{4}P(G\setminus B_t)\\
    &\geqslant \frac{C(d)}{4}\abs{G\setminus B_t}^{1-1/d},
\end{align*}
by the isoperimetric inequality.
Thus, for almost every $t\in [r,1]$,
\begin{align*}
    \frac{\mathrm{d}}{\diff{t}}\abs{G\setminus B_t}=-\frac{\diff{}}{\diff{t}}\abs{G\cap B_t}
    \leqslant -\frac{C(d)}{4}\abs{G\setminus B_t}^{1-1/d}.
\end{align*}

By Gronwall's inequality, for all $t\in [r,1]$,
\begin{align*}
\abs{G\setminus B_t}^{1/d}&\leqslant \max\{0, \abs{G\setminus B_r}^{1/d}-\frac{C(d)}{4d}(t-r)\} \\ 
&\stackrel{\eqref{eq: volume_rate}}{\leqslant} \max\{0, \left(C(d,p)\abs{G}r^{\alpha/2}\right)^{1/d}-\frac{C(d)}{4d}(t-r)\} \\
&=\max\{0, \left(C(d,p)w_d \right)^{1/d}r^{1+\alpha/{2d}}-\frac{C(d)}{4d}(t-r)\}.
\end{align*}
In particular,
\[\abs{G\setminus B_1}^{1/d}\leqslant\max\{0,  \left(C(d,p)w_d \right)^{1/d}r^{1+\alpha/{2d}}-\frac{C(d)}{4d}(1-r)\}=0\]
whenever $r$ is sufficiently small.
Hence, for $r$ sufficiently small, it holds that
$\abs{G\setminus B_1}=0$, and the set $E=G$ satisfies \eqref{eq: comp_bdset}.

\end{proof}

Thanks to \prettyref{thm: comp_bdset}, we are able to apply the direct method to prove \prettyref{thm: minTset}.

\begin{proof}[Proof of \prettyref{thm: minTset}]
Let $(G_k)$ be a minimizing sequence to problem (\ref{problem: T}) with each $G_k$ being a bounded subset of $\R^d$ and $\abs{G_k}=m$. By \prettyref{thm: comp_bdset}, there exists an alternating minimizing sequence $(E_k)$ to problem (\ref{problem: T}) with $\abs{E_k}=m$, which is uniformly bounded by $B_2$. By the compactness of bounded sets of finite perimeter (Proposition \prettyref{prop: compact_bdset}), there exists a set $E$ of finite perimeter in $\R^d$, such that up to extracting a subsequence if necessary:
\[E_k\to E\quad\textrm{and}\quad E\subseteq B_2.\]

By the lower semi-continuity of $\W_p$ (\prettyref{lem: lsc_Wp}) and the lower semi-continuity of perimeter (Proposition \prettyref{prop: lsc of perimeter}), $T$ is lower semi-continuous. Thus,
\[T(E)\leqslant \liminf_{k\to\infty} T(E_k),\]
which yields that $E$ is a minimizer to problem (\ref{problem: T}).
\end{proof}

\nocite{*}
\bibliographystyle{alpha}
\bibliography{cite}

\Addresses
\end{document}

%% file: separate.tikz
\begin{figure}[h]
\centering
\begin{tikzpicture}

\draw[fill=gray!50!white] (8.5,0) circle (1.2);

\draw[fill=gray!50!white] plot [smooth,tension=0.5] coordinates {(8.5,1.2) (8.3,2) (8.21,1.16)};

\draw[red,line width=0.5mm] (8.5,1.2) arc (90:105:1.2);

\node at (8.2,1.5) {$G_2$};
\node at (8.5,0) {$G_1$};

\end{tikzpicture}
\caption[]{If a set $G$ can be split into a dominated part $G_1$ and a remainder part $G_2$, with a small slicing surface area bounded by $\frac{1}{4}T(G_2)$, then $G$ may not be a $T$-minimizer.}
\label{fig:separate}
\end{figure}
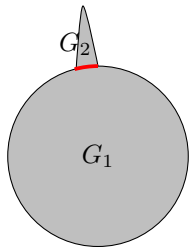